\documentclass[10pt, a4paper]{article}
\usepackage{amsmath, amsfonts}
\usepackage[usenames]{color}
\usepackage[latin1]{inputenc}
\usepackage[table, dvipsnames]{xcolor}
\usepackage{array}
\newcolumntype{P}[1]{>{\centering\arraybackslash}p{#1}}
\newcolumntype{M}[1]{>{\centering\arraybackslash}m{#1}}
\usepackage{multirow}
\usepackage{mathrsfs}
\usepackage{hyperref} 
\hypersetup{bookmarks = true, 
                  colorlinks = true, 
                  linkcolor = blue, 
                  citecolor = blue, 
                  }
\newcommand{\xdownarrow}[1]{%
  {\left\downarrow\vbox to #1{}\right.\kern-\nulldelimiterspace}
}

\newcommand{\xuparrow}[1]{%
  {\left\uparrow\vbox to #1{}\right.\kern-\nulldelimiterspace}
}

\usepackage{graphicx}
\usepackage{amssymb}
\usepackage{amsmath}

\newtheorem{theorem}{Theorem}[section]

\newtheorem{corollary}[theorem]{Corollary}

\newtheorem{definition}[theorem]{Definition}
\newtheorem{example}[theorem]{Example}

\newtheorem{lemma}[theorem]{Lemma}

\newtheorem{remark}[theorem]{Remark}

\newenvironment{proof}[1][Proof]{\noindent\textbf{#1.} }{\ \rule{0.5em}{0.5em}}

\title{\textbf{Equisingularity of families of map germs from the plane to $3$-space}}
\author{ \ \ \ \\{O.N. Silva\footnote{O.N. Silva: Departamento de Matemática, Universidade Federal da Paraíba (UFPB), PB, Brazil.  \hspace{5cm} e-mail: otoniel.silva@academico.ufpb.br}}}
\date{}

\begin{document}

\maketitle

\begin{abstract}
In this work, we consider a finitely determined map germ $f$ from $(\mathbb{C}^2,0)$ to $(\mathbb{C}^3,0)$. We characterize the Whitney equisingularity of an unfolding $F=(f_t,t)$ of $f$ through the constancy of a single invariant in the source. Namely, the Milnor number of the curve $W_t(f)=D(f_t)\cup f_t^{-1}(\gamma)$, where $D(f)$ denotes the double point curve of $f_t$. This gives an answer to a question by Ruas in $1994$.
\end{abstract}

\section{Introduction}

$ \ \ \ \ $ Throughout this paper, we assume that $f:(\mathbb{C}^2,0)\rightarrow(\mathbb{C}^3,0)$ is a finite, generically $1-1$, holomorphic map germ, unless otherwise stated. We consider the Whitney equisingularity of an unfolding $F=(f_t,t)$ of $f$, where $f$ is a finitely determined map germ from $(\mathbb{C}^2,0)$ to $(\mathbb{C}^3,0)$. 

For finitely determined map germs from $(\mathbb{C}^n,0)$ to $(\mathbb{C}^p,0)$, $n<p$, in the nice dimensions or their boundary, Gaffney proved (\cite[Theorem 7.3]{gaffney}) that an excellent unfolding $F$ is Whitney equisingular if, and only if, all polar invariants of all stable types of $F$ are constant. The problem is that there is a big amount of invariants to consider. However, most of the times these invariants are related and it is possible to reduce this number and to remove the condition for $F$ being excellent. This was done for example in \cite{victor}, in the case of map germs from $\mathbb{C}^3$ to $\mathbb{C}^4$, and in \cite{juanjoevictor}, in the case of map germs from $\mathbb{C}^n$ to $\mathbb{C}^{2n-1}$, $n\geq 3$. In our case, a natural question is:\\

\noindent \textbf{Question 1:} Let $f:(\mathbb{C}^2,0)\rightarrow (\mathbb{C}^3,0)$ be a finitely determined map germ and $F$ an unfolding of $f$. What is the minimum number of invariants we need to characterize the Whitney equisingularity of $F$?\\ 

In 1993, Gaffney showed that this minimum number is less than or equal to two. More precisely, he showed that the Whitney equisingularity of $F$ is equivalent to the constancy of the invariant $e_D(f_t)$, introduced by Gaffney in \cite{gaffney}, and the first polar multiplicity $m_1(f_t(\mathbb{C}^2))$ (see \cite[Theorem 8.7]{gaffney}). Moreover, if $f$ has corank $1$, then the constancy of $m_1(f_t(\mathbb{C}^2))$ is not necessary. We remark that the invariant $e_D(f)$ is defined as

\begin{center}
$e_D(f):=dim_{\mathbb{C}}\dfrac{\mathcal{O}_2}{\langle\lambda, J[\lambda, p_1 \circ f] \rangle}$
\end{center}

\noindent where $\lambda$ defines $D(f)$, $p_1$ is a generic projection from $\mathbb{C}^3$ to $\mathbb{C}$ and $J[\lambda, p_1 \circ f]$ denotes the Jacobian ideal of the map germ from $\mathbb{C}^2$ to $\mathbb{C}^2$ defined by $\lambda$ and $p_1 \circ f$. Note that the invariant $e_D(f)$ is calculated in the source of $f$. On the other hand, the invariant $m_1(f(\mathbb{C}^2))$ is related with the target of $f$. So another related question is:\\



\noindent \textbf{Question 2:} Let $f$ and $F$ as in Question 2. Can the Whitney equisingularity of $F$ be controlled by the constancy of invariants calculated only in the source? More specifically, the Milnor number of a single curve in the source?\\ 

In 1994, Ruas conjectured that Question $2$ has a positive answer. For Question $1$, she also conjectured that the minimum number of invariants needed to control the Whitney equisingularity of $F$ is one. More precisely, she conjectured that $F$ is Whitney equisingular if and only if $\mu(D(f_t))$ is constant (\cite{ruas} see also \cite[Conjecture 1.2]{otoniel1}). In 2018, almost $25$ years later, Ruas and the author presented examples where $\mu(D(f_t))$ is constant and $F$ is not Whitney equisingular, showing that Ruas's conjecture is not true (see \cite{otoniel1} and \cite{otoniel3}).

In 2012, Marar, Nuño-Ballesteros y Peñafort-Sanchis showed that $F$ is Whitney equisingular if and only if the Milnor numbers of $D(f_t)$ and $\gamma_t$ are constant, where $\gamma_t$ is the intersection of $f_t(\mathbb{C}^2)$ with a generic plane in $\mathbb{C}^3$, the transversal slice of $f_t(\mathbb{C}^2)$ (see \cite[Theorem 5.3]{ref9}). However, with the examples presented in \cite{otoniel1}, it seems that the invariants $\mu(D(f))$ and $\mu(\gamma_t,0)$ are independent, in the sense that the constancy of one does not imply the constancy of the other and vice versa.

In this paper, we provide answers to Questions 1 and 2. More precisely, we define a new invariant, namely, the Milnor number of the curve $W:=D(f) \cup f^{-1}(\gamma)\subset \mathbb{C}^2$ and we prove the following result:  

\begin{theorem}\label{main result 3} Let $f:(\mathbb{C}^2,0)\rightarrow(\mathbb{C}^3,0)$ be a finitely determined map germ and let $F:(\mathbb{C}^2 \times \mathbb{C},0)\rightarrow (\mathbb{C}^3 \times \mathbb{C},0)$, $F=(f_t,t)$, be an unfolding of $f$. Set $W(f_t):=D(f_t)\cup f_t^{-1}(\gamma_t)$, where $\gamma_t$ is the transversal slice of $f_t(\mathbb{C}^2)$. Then

\begin{center}
$F$ is Whitney equisingular $ \ \ \ $ $\Leftrightarrow$ $ \ \ \ $ $\mu(W(f_t),0)$ is constant.
\end{center}
\end{theorem}

Note that Theorem \ref{main result 3} means that part of Ruas's conjecture is true, that is, the minimum number mentioned in Question 1 is in fact one and Question 2 has a positive answer. Since all branches of the curve $D(f)$ also are branches of the curve $W(f)$, we see that Ruas's conjecture was not so far from the characterization of Whitney equisingularity through the constancy of a single invariant in the source.

\section{Preliminaries}

$ \ \ \ \ $ Throughout this paper, given a finite map $f:\mathbb{C}^2\rightarrow \mathbb{C}^3$, $(x,y)$ and $(X,Y,Z)$ are used to denote systems of coordinates in $\mathbb{C}^2$ (source) and $\mathbb{C}^3$ (target), respectively. Also, $\mathbb{C} \lbrace x_1,\cdots,x_n \rbrace \simeq \mathcal{O}_n$ denotes the local ring of convergent power series in $n$ variables. Throught, we use the standard notation of singularity theory as the reader can find in Wall's survey paper \cite{wall}.

We recall the notion of $\mathcal{A}$-finite determinacy, where $\mathcal{A}$ is the group of coordinates change in the source and in the target, as defined by Mather in \cite{wall}. We say that two map germs $f,g:(\mathbb{C}^n,0)\rightarrow (\mathbb{C}^{p},0)$ are $\mathcal{A}$-equivalent, denoted by $g\sim_{\mathcal{A}}f$, if there exist germs of diffeomorphisms $\eta:(\mathbb{C}^n,0)\rightarrow (\mathbb{C}^n,0)$ and $\xi:(\mathbb{C}^{p},0)\rightarrow (\mathbb{C}^{p},0)$, such that $g=\xi \circ f \circ \eta$. 

Consider a finite map germ $f:(\mathbb{C}^2,0)\rightarrow (\mathbb{C}^3,0)$. By Mather-Gaffney criterion (\rm\cite[Theorem 2.1]{wall}), $f$ is finitely determined if and only if there is a finite representative $f:U \rightarrow V$, where $U\subset \mathbb{C}^2$, $V \subset \mathbb{C}^3$ are open neighbourhoods of the origin, such that $f^{-1}(0)=\lbrace 0 \rbrace$ and the restriction $f:U \setminus \lbrace 0 \rbrace \rightarrow V \setminus \lbrace 0 \rbrace$ is stable. This means that the only singularities of $f$ on $U \setminus \lbrace 0 \rbrace$ are cross-caps (or Whitney umbrellas), transverse double and triple points. By shrinking $U$ if necessary, we can assume that there are no cross-caps nor triple points in $U$. With this point of view, the set of singularities that will be important in this setting is the set of double points of $f$.

\subsection{Double point spaces}

Multiple point spaces of a map germ from $(\mathbb{C}^n,0)$ to $(\mathbb{C}^p,0)$ with $n\leq p$ play an important role in the study of its geometry. In this work, we will deal only with double points, which will be fundamental to control the finiteness determinacy of $f$. We describe the sets of double points of a finite map from $\mathbb{C}^2$ to $\mathbb{C}^3$ following the description given in \cite[Section 2]{ref9} (and also in \cite[Section 1]{ref7} and \cite[Section 3]{ref12}). Let $U \subset \mathbb{C}^2$ and $V \subset \mathbb{C}^3$ be open sets. Throughout we assume that a map $f:U\rightarrow V$ is finite, that is, holomorphic, closed and finite-to-one, unless otherwise stated.

Following Mond \cite[Section 3]{ref12}, we define spaces related to the double points of a given finite mapping $f:U \rightarrow V$, by firstly considering the sheaf $\mathcal{I}_{2}$ and $\mathcal{I}_{3}$  defining the diagonal of $\mathbb{C}^2 \times \mathbb{C}^2$ and $\mathbb{C}^3 \times \mathbb{C}^3$, respectively. That is, locally 

\begin{center}
 $\mathcal{I}_{2}=\langle x-x^{'},y-y^{'} \rangle$ $ \ \ \ $ and $ \ \ \ $ $\mathcal{I}_{3}=\langle X-X^{'},Y-Y^{'}, Z-Z^{'} \rangle$.
 \end{center} 

Since the pull-back $(f \times f)^{\ast}\mathcal{I}_{3}$ is contained in $\mathcal{I}_{2}$ and $U$ is small enough, then there exist sections $\alpha_{ij}\in \mathcal{O}_{\mathbb{C}^{2} \times \mathbb{C}^2}(U \times U)$ well defined in all $U \times U$, such that

\begin{center}
$f_{i}(x,y)-f_{i}(x^{'},y^{'})= \alpha_{i1}(x,y,x^{'},y^{'})(x-x^{'})+ \alpha_{i2}(x,y,x^{'},y^{'})(y-y^{'}), \ for \ i=1,2,3.$
\end{center}

If $f(x,y)=f(x^{'},y^{'})$ and $(x,y) \neq (x^{'},y^{'})$, then every $2 \times 2$ minor of the matrix $\alpha=(\alpha_{ij})$ must vanish at $(x,y,x^{'},y^{'})$. Denote by $\mathcal{R}(\alpha)$ the ideal in $\mathcal{O}_{\mathbb{C}^{4}}$ generated by the $2\times 2$ minors of $\alpha$. Then we have the following definition.

\begin{definition} The \textit{lifting of the double point space of $f$} is the complex space

\[
D^{2}(f)=V((f\times f)^{\ast}\mathcal{I}_{3}+\mathcal{R}(\alpha)).
\]

\end{definition}

Although the ideal $\mathcal{R}(\alpha)$ depends on the choice of the coordinate functions of $f$, in \cite[Section 3]{ref12} it is proved that the ideal $(f\times f)^{\ast}\mathcal{I}_{3}+\mathcal{R}(\alpha)$ does not, and so $D^{2}(f)$ is well defined. It is not hard to see that the points in the underlying set of $D^{2}(f)$ are exactly the ones in $U \times U$ of type $(x,y,x^{'},y^{'})$ with $(x,y) \neq (x^{'},y^{'})$, $f(x,y)=f(x^{'},y^{'})$ and of type $(x,y,x,y)$ such that $(x,y)$ is a singular point of $f$.

Let $f:(\mathbb{C}^2,0)\rightarrow(\mathbb{C}^{3},0)$ be a finite map germ and denote by $I_{3}$ and $R(\alpha)$ the stalks at $0$ of $\mathcal{I}_{3}$ and $\mathcal{R}(\alpha)$. By taking a representative of $f$, we define the \textit{lifting of the double point space of the map germ $f$} as the complex space germ $D^{2}(f)=V((f \times f)^{\ast}I_{3}+R(\alpha))$.

Once the lifting $D^2(f) \subset \mathbb{C}^2 \times \mathbb{C}^2$ is defined, we now consider its image $D(f)$ on $\mathbb{C}^2$ by the projection $\pi:\mathbb{C}^2 \times \mathbb{C}^2 \rightarrow \mathbb{C}^2$ onto the first factor. The most appropriate structure for $D(f)$ is the one given by the Fitting ideals, because it relates in a simple way the properties of the spaces $D^2(f)$ and $D(f)$. 

Another important space to study the topology of $f(\mathbb{C}^{2})$ is the double point curve in the target, that is, the image of $D(f)$ by $f$, denoted by $f(D(f))$, which will also be consider with the structure given by Fitting ideals. 

More precisely, given a finite morphism of complex spaces $f:X\rightarrow Y$ the push-forward $f_{\ast}\mathcal{O}_{X}$ is a coherent sheaf of $\mathcal{O}_{Y}-$modules (see \cite[Chapter 1]{grauert}) and to it we can (as in \rm\cite[Section 1]{ref13}) associate the Fitting ideal sheaves $\mathcal{F}_{k}(f_{\ast}\mathcal{O}_{X})$. Notice that the support of $\mathcal{F}_{0}(f_{\ast}\mathcal{O}_{X})$ is just the image $f(X)$. Analogously, if $f:(X,x)\rightarrow(Y,y)$ is a finite map germ then we denote also by $ \mathcal{F}_{k}(f_{\ast}\mathcal{O}_{X})$ the \textit{k}th Fitting ideal of $\mathcal{O}_{X,x}$ as $\mathcal{O}_{Y,y}-$module. In this way, we have the following definition:

\begin{definition} Let $f:U \rightarrow V$ be a finite mapping.\\

\noindent (a) Let ${\pi}|_{D^2(f)}:D^2(f) \subset U \times U \rightarrow U$ be the restriction to $D^2(f)$ of the projection $\pi$. The \textit{double point space} is the complex space

\begin{center}
$D(f)=V(\mathcal{F}_{0}({\pi}_{\ast}\mathcal{O}_{D^2(f)}))$.
\end{center}

\noindent Set theoretically we have the equality $D(f)=\pi(D^{2}(f))$.\\

\noindent (b) The \textit{double point space in the target} is the complex space $f(D(f))=V(\mathcal{F}_{1}(f_{\ast}\mathcal{O}_2))$. Notice that the underlying set of $f(D(f))$ is the image of $D(f)$ by $f$.\\ 

\noindent (c) Given a finite map germ $f:(\mathbb{C}^{2},0)\rightarrow (\mathbb{C}^3,0)$, \textit{the germ of the double point space} is the germ of complex space $D(f)=V(F_{0}(\pi_{\ast}\mathcal{O}_{D^2(f)}))$. \textit{The germ of the double point space in the target} is the germ of the complex space $f(D(f))=V(F_{1}(f_{\ast}\mathcal{O}_2))$.

\end{definition}

\begin{remark} If $f:U \subset \mathbb{C}^2 \rightarrow V \subset \mathbb{C}^3 $ is finite and generically $1$-to-$1$, then $D^2(f)$ is Cohen-Macaulay and has dimension $1$ (see \rm\cite[\textit{Prop. 2.1}]{ref9}\textit{). Hence, $D^2(f)$, $D(f)$ and $f(D(f))$ are curves. In this case, without any confusion, we also call these complex spaces by the ``lifting of the double point curve'', the ``double point curve'' and the ``image of the double point curve'', respectively.}
\end{remark}

We note that the set $D(f)$ plays a fundamental role in the study of the finite determinacy. In \cite[Theorem 2.14]{ref7}, Marar and Mond presented necessary and sufficient conditions for a map germ $f:(\mathbb{C}^n,0)\rightarrow (\mathbb{C}^p,0)$ with corank $1$ to be finitely determined in terms of the dimensions of $D^2(f)$ and other multiple points spaces. In \cite{ref9}, Marar, Nu\~{n}o-Ballesteros and Pe\~{n}afort-Sanchis extended in some way this criterion of finite determinacy to the corank $2$ case. More precisely, they proved the following result:

\begin{theorem}\rm(\cite[Corollary 3.5]{ref9})\label{criterio} \textit{
Let $f:(\mathbb{C}^2,0)\rightarrow(\mathbb{C}^{3},0)$ be a finite and generically $1$ $-$ $1$ map germ. Then $f$ is finitely determined if and only if $\mu(D(f))$ is finite (equivalently, $D(f)$ is a reduced curve).}
\end{theorem}

\subsection{Identification and Fold components of $D(f)$}\label{foldcomp}


$ \ \ \ \ $ When $f:(\mathbb{C}^2,0)\rightarrow (\mathbb{C}^3,0)$ is finitely determined, the restriction of (a representative) $f$ to $D(f)$ is finite. In this case, $f_{|D(f)}$ is generically $2$-to-$1$ (i.e; $2$-to-$1$ except at $0$). On the other hand, the restriction of $f$ to an irreducible component $D(f)^i$ of $D(f)$ can be generically $1$-to-$1$ or $2$-to-$1$. This motivates us to give the following definition which is from \cite[Definition 4.1]{otoniel3} (see also \cite[Definition 2.4]{otoniel1}).

\begin{definition}\label{typesofcomp} Let $f:(\mathbb{C}^2,0)\rightarrow (\mathbb{C}^3,0)$ be a finitely determined map germ. Let $f:U\rightarrow V$ be a representative, where $U$ and $V$ are neighbourhoods of $0$ in $\mathbb{C}^2$ and $\mathbb{C}^3$, respectively, and consider an irreducible component $D(f)^j$ of $D(f)$.\\

\noindent (a) If the restriction ${f_|}_{D(f)^j}:D(f)^j\rightarrow V$ is generically $1$ $-$ $1$, we say that $D(f)^j$ is an \textit{identification component of} $D(f)$.\\ 

In this case, there exists an irreducible component $D(f)^i$ of $D(f)$, with $i \neq j$, such that $f(D(f)^j)=f(D(f)^i)$. We say that $D(f)^i$ is the \textit{associated identification component to} $D(f)^j$ or that the pair $(D(f)^j, D(f)^i)$ is a \textit{pair of identification components of} $D(f)$.  \\

\noindent (b) If the restriction ${f_|}_{D(f)^j}:D(f)^j\rightarrow V$ is generically $2$ $-$ $1$, we say that $D(f)^j$ is a \textit{fold component of} $D(f)$.\\

\noindent (c) We define the sets $IC(D(f))=\lbrace $identification components of $D(f) \rbrace $ and  $FC(D(f))= \lbrace \  $fold components of $D(f) \rbrace $. And we define the numbers $r_i:=\sharp IC(D(f))$ and $r_f:= \sharp FC(D(f))$.  
\end{definition}

The following example illustrates the two types of irreducible components of $D(f)$ presented in Definition \ref{typesofcomp}.

\begin{example}\label{example}
\textit{Let $f(x,y)=(x,y^2,xy^3-x^5y)$ be the singularity $C_5$ of Mond's list} \rm(\cite[\textit{p.378}]{mond6}). \textit{In this case, $D(f)=V(xy^2-x^5)$. Then $D(f)$ has three irreducible components given by}

\begin{center}
$D(f)^1=V(x^2-y), \ \ \ $ $D(f)^2=V(x^2+y) \ \ $ and $ \ \ D(f)^3=V(x)$. 
\end{center}

\textit{Notice that $D(f)^3$ is a fold component and $(D(f)^1$, $D(f)^2)$ is a pair of identification components. Also, we have that $f(D(f)^3)=V(X,Z)$ and $f(D(f)^1)=f(D(f)^2)=V(Y-X^4,Z)$ (see Figure \rm\ref{figura38}).}\\

\begin{figure}[h]
\centering
\includegraphics[scale=0.3]{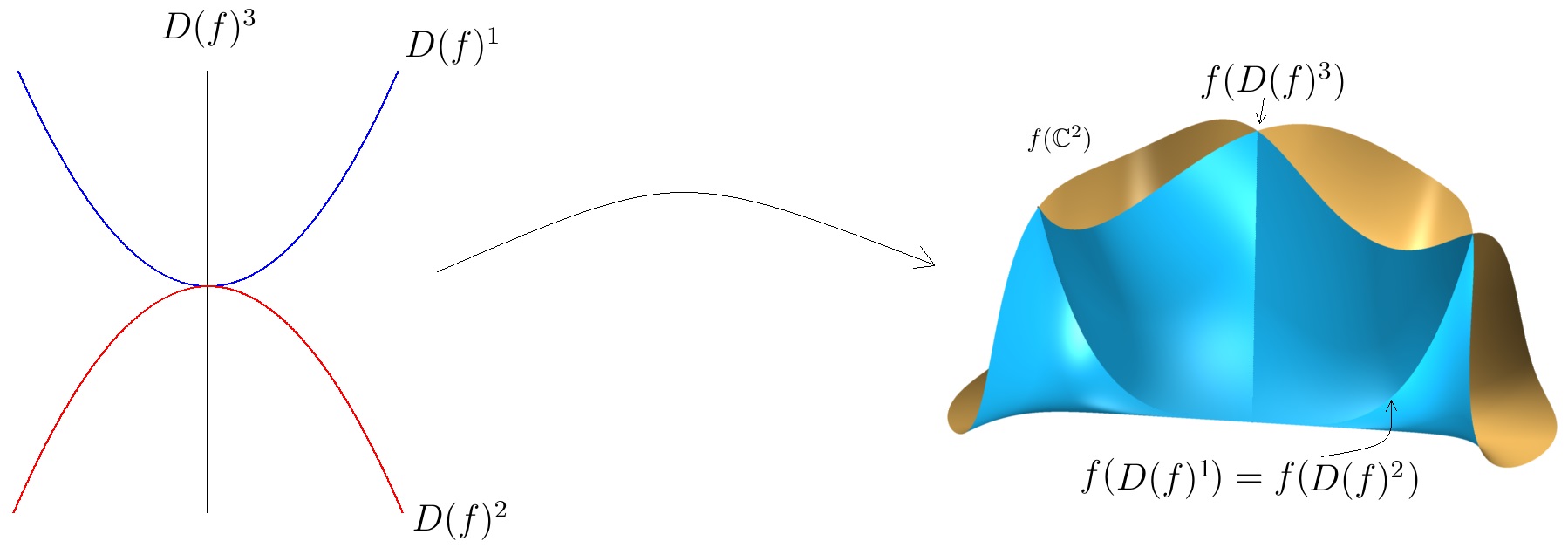} 
\caption{Identification and fold components of $D(f)$ (real points)}\label{figura38}
\end{figure}

\end{example}

\section{Whitney equisingularity}

$ \ \ \ \ $ Let $f:(\mathbb{C}^2,0)\rightarrow (\mathbb{C}^3,0)$ be a finitely determined map germ. If $H$ is a generic plane in $\mathbb{C}^3$, then we call the curve $\gamma(f):=f(\mathbb{C}^2) \cap H$ by transversal slice of $f(\mathbb{C}^2)$. The pre-image of $\gamma(f)$ by $f$ will be denoted by $\tilde{\gamma}(f)$. The intersection multiplicity of two planes curves $C^1$ and $C^2$ at $0$ will be denoted by $i(C^1,C^2)$. We recall that $i(C^1,C^2)\geq m(C^1,0)\cdot m(C^2,0)$, the product of the multiplicities of $(C^1,0)$ and $(C^2,0)$. Now, we define our new invariant as the Milnor number of the plane curve $W(f)$ defined as

\begin{center}
$W(f):=D(f) \cup \tilde{\gamma}(f)$.
\end{center}

The following lemma shows how the Milnor number is related with other invariants of $f$.

\begin{lemma}\label{lemmaaux4} Let $f:(\mathbb{C}^2,0)\rightarrow (\mathbb{C}^3,0)$ be a finitely determined map germ. Then\\

\noindent (a) $i(D(f),\tilde{\gamma}(f))=2 m(f(D(f))$.

\noindent (b) $m(D(f))\cdot m(\tilde{\gamma}) \leq 2m(f(D(f)))$. In particular, if $f$ has corank $1$, then $m(D(f)) \leq 2m(f(D(f)))$.

\noindent (c) $\mu(D(f))+\mu(\tilde{\gamma})+4m(f(D(f))-1$.

\end{lemma}

\begin{proof} (a) Let $(X,Y,Z)$ be the coordinates of $\mathbb{C}^3$. After a change of coordinates, we can suppose that the plane $H=V(X)$ is generic for $f=(f_1,f_2,f_3)$, that is, $\gamma(f)=f(\mathbb{C}^2) \cap V(X)$. By the genericity of $H$, we can suppose that 

\begin{equation}\label{eq11}
C(f(D(f))) \cap H= \lbrace 0 \rbrace.
\end{equation}

\noindent where $C(f(D(f)))$ denotes the Zariski tangent cone of $f(D(f))$. 

This means that if $v=(v_1,v_2,v_3)$ is the direction of the tangent vector of an irreducible component of $f(D(f))$, then $v_1 \neq 0$. Let $D(f)^i$ be an irreducible component of $D(f)$ and consider its image $f(D(f)^i)$ by $f$. Consider a parametrization $\varphi_i:W\rightarrow D(f)^i$ of $D(f)^i$ defined by

\begin{center}
$\varphi_{i}(u)=(\varphi_{i,1}(u),\varphi_{i,2}(u))$,
\end{center}

\noindent where $W$ is a neighboorhood of $0$ in $\mathbb{C}$ and $\varphi_{i,1}(u), \varphi_{i,1}(u)\in \mathbb{C}\lbrace u \rbrace$. 

Consider the mapping $\hat{\varphi}_{i}: W \rightarrow f(D(f)^i)$, defined by  

\begin{equation}\label{eq12}
\hat{\varphi}_{i}(u):=(f_1(\varphi_i(u)),f_2(\varphi_i(u)),f_3(\varphi_i(u))).
\end{equation}

Suppose that $D(f)^i$ is an identification component of $D(f)$. Since the restriction of $f$ to $D(f)^i$ is generically $1$-to-$1$, the mapping $\hat{\varphi}_{i}$ in (\ref{eq12}) is a parametrization of $f(D(f)^i)$ (a primitive one). On the other hand, if $D(f)^i$ is a fold component of $D(f)$ then the restriction of $f$ to $D(f)^i$ is $2$-to-$1$, hence the mapping $\hat{\varphi}_{i}$ in (\ref{eq12}) is generically $2$-to-$1$ and it is also a parametrization of $f(D(f)^i)$ (a double cover one). 

The condition (\ref{eq11}) implies that $ord_u(f_1(\varphi_i(u)))=m(f(D(f)^i))$ in the case where $D(f)^i$ is an identification component of $D(f)$. On the other hand, $ord_u(f_1(\varphi_i(u)))=2m(f(D(f)^i))$ if $D(f)^i$ is a fold component one, where $ord_u(p(u))$ denotes the degree of the non-zero term of lowest degree of $p(u) \in \mathbb{C}\lbrace u \rbrace$. Since we choose $H$ as the plane $X=0$, we have that $\tilde{\gamma}:=V(f_1(x,y)) \subset \mathbb{C}^2$. Hence we have that 

\begin{center}
$i(D(f)^i,\tilde{\gamma})=ord_u f_1(\varphi_{i,1}(u),\varphi_{i,2}(u))=ord_u f_1(\varphi_i(u))$,
\end{center}

\noindent see for instance \cite[page 174]{greuellivro}. If $D(f)^i \in IC(D(f))$, then $ord_u f_1(\varphi_i(u))=m(f(D(f)^i)$. On the other hand, if $D(f)^i \in FC(D(f))$ then $ord_u f_1(\varphi_i(u))=2m(f(D(f)^i)$. Suppose that there are $2k$ irreducible components of $D(f)$ in $IC(D(f))$ and let $(D(f)^{i_1},D(f)^{j_1})$, $\cdots$, $(D(f)^{i_{k}},D(f)^{j_{k}})$ be all pairs of identifications components of $D(f)$. Let $D(f)^{s_1},\cdots,D(f)^{s_p}$ be all fold components of $D(f)$. Then

\begin{center}
$i(D(f), \tilde{\gamma}(f))= \left( \displaystyle \sum_{q=1}^{k} 2 \cdot m(f(D(f)^{i_q})) \right)+ \left( \displaystyle \sum_{q=1}^{p} 2 \cdot m(f(D(f)^{s_q})) \right)=2m(f(D(f)))$.
\end{center}

\noindent (b) It follows by (a) and the property that $i(C^1,C^2)\geq m(C^1,0)\cdot m(C^2,0)$.\\

\noindent (c) Let $W^1,\cdots,W^{r+s}$ be the irreducible components of $W=D(f) \cup \tilde{\gamma}$, which is a disjoint union. Hence, $W^i$ is an irreducible component either of $D(f)$ or $\tilde{\gamma}$. It is convenient to give a different notation for branches in $D(f)$ and $\tilde{\gamma}$. Let $\tilde{\gamma}^1, \cdots, \tilde{\gamma}^s$ and $D(f)^1,\cdots, D(f)^r$ be the irreducible components of $\tilde{\gamma}$ and $D(f)$, respectively. By \cite[Corollary 1.2.3]{buch} and (a), we have that

\begin{center}
$\mu(W,0)= \displaystyle \sum_{q=1}^{r+s} (\mu((W^q)-1) \ + 2 \left( \displaystyle \sum_{q<j} i(W^q,W^j) \right) +1$ .
\end{center}

However, we have that the intersection multiplicity for plane curves has the additive property. Hence, we have that

\begin{flushleft}
$\mu(W,0)=$\\
$= \displaystyle \sum_{q=1}^r (\mu(D(f)^q)-1)+ \displaystyle \sum_{q=1}^s (\mu(\tilde{\gamma}^q)-1) + \displaystyle 2 \left(\sum_{1<q,j<r} i(D(f)^q,D(f)^j)+\displaystyle \sum_{1<q,j<r} i(\tilde{\gamma}^q,\tilde{\gamma}^j)+\displaystyle  \sum_{q=1}^s \sum_{j=1}^r i(\tilde{\gamma}^q,D(f)^j) \right)+1$\\

$= \left( \displaystyle \sum_{q=1}^r (\mu(D(f)^q)-1)+ \displaystyle 2 \left(\sum_{1<q,j<r} i(D(f)^q,D(f)^j)\right)+1 \right) + \left( \displaystyle \sum_{q=1}^s (\mu(\tilde{\gamma}^q)-1) + \displaystyle 2 \left( \sum_{1<q,j<r} i(\tilde{\gamma}^q,\tilde{\gamma}^j)\right) + 1 \right)$\\

$ \ \ $\\

$ \ \ \ \  + \displaystyle 2 i(D(f), \tilde{\gamma}) -1$\\

$ \ \ $\\

$= \mu(D(f))+\mu(\tilde{\gamma})+4m(f(D(f))-1$,
\end{flushleft}

\noindent where the last equality follows by (a) and each Milnor number is calculated at the origin.\end{proof}\\

For convenience of the reader, we write Theorem \ref{main result 3} again.

\begin{theorem}\label{main result 3.1} Let $f:(\mathbb{C}^2,0)\rightarrow(\mathbb{C}^3,0)$ be a finitely determined map germ and let $F:(\mathbb{C}^2 \times \mathbb{C},0)\rightarrow (\mathbb{C}^3 \times \mathbb{C},0)$, $F=(f_t,t)$, be an unfolding of $f$. Set $W(f_t):=D(f_t)\cup f_t^{-1}(\gamma_t)$, where $\gamma_t$ is the transversal slice of $f_t(\mathbb{C}^2)$. Then

\begin{center}
$F$ is Whitney equisingular $ \ \ \ $ $\Leftrightarrow$ $ \ \ \ $ $\mu(W(f_t),0)$ is constant.
\end{center}
\end{theorem}

\begin{proof} By Lemma \ref{lemmaaux4}(b), we have that 

\begin{equation}\label{eq10}
\mu(W_t,0)=\mu(D(f_t),0)+4m(f_t(D(f_t)))+\mu(\tilde{\gamma}(f_t),0)-1
\end{equation}

\noindent Now the result follows by \cite[Theorem 5.3]{ref9} and the upper semi-continuity of the invariants in (\ref{eq10}).\end{proof}\\

An interesting case is when $\tilde{\gamma}$ is transversal to $D(f)$. In this case, the intersection multiplicity $i(D(f),\tilde{\gamma})$ is simply the product of the multiplicities of $D(f)$ and $\tilde{\gamma}$. This case is explored in the following corollary.

\begin{corollary}\label{app1} Let $f:(\mathbb{C}^2,0)\rightarrow(\mathbb{C}^3,0)$ be a finitely determined map germ and let $F:(\mathbb{C}^2 \times \mathbb{C},0)\rightarrow (\mathbb{C}^3 \times \mathbb{C},0)$, $F=(f_t,t)$, be an unfolding of $f$. Suppose now that $\tilde{\gamma}$ is transversal to $D(f)$. Then\\

\noindent (a) $2m(f(D(f))= m(D(f))\cdot m(\tilde{\gamma})$.\\

In particular, if $f$ has corank $1$, then $m(D(f))=2m(f(D(f)))$.\\

\noindent (b) If $F$ is topologically trivial, then $m(f_t(D(f_t)))$ is constant. In addition, if $f$ has corank $1$, then $F$ is Whitney equisingular.

\end{corollary}

\begin{corollary} If $f:(\mathbb{C}^2,0)\rightarrow (\mathbb{C}^3,0)$ is a corank $2$ finitely determined double fold map germ, i.e. a map germs of the form $f(x,y)=(x^2,y^2,h(x,y))$, then

\begin{center}
 $m(D(f))=m(f(D(f)))$.
\end{center}

\end{corollary}

\begin{flushleft}
\textit{Acknowlegments:} We would like to thank Ruas and Nuño-Ballesteros for many helpful conversations, suggestions and comments on this work.
\end{flushleft}

\end{document}